\documentclass[12pt]{amsart}
\usepackage{amssymb,latexsym,amsmath,enumerate,nicefrac}

\usepackage{amsaddr}
\usepackage[total={6.5in,9in}, top=1in, left=1in, headsep = 0in]{geometry}

\usepackage{fancyhdr}
\usepackage{color}
\definecolor{myurlcolor}{rgb}{0.6,0,0}
\definecolor{mycitecolor}{rgb}{0,0,0.8}
\definecolor{myrefcolor}{rgb}{0,0,0.8}
\usepackage[bookmarks=false,pagebackref=true]{hyperref}
\hypersetup{colorlinks,
linkcolor=myrefcolor,
citecolor=mycitecolor,
urlcolor=myurlcolor}

\renewcommand*{\backref}[1]{}

\usepackage[all]{xy}
\usepackage{tikz}
\usetikzlibrary{calc}
\usepackage{graphics}
\usepackage{dsfont}
\usepackage{amsfonts}
\usepackage{mathtools}
\usepackage{comment}

\pgfkeys{/tikz/.cd, arrowhead/.default=-1pt, arrowhead/.code={}}

\newcommand{\dwsp}{\hspace{1pt}}

\allowdisplaybreaks

\usepackage{graphics}
\usepackage{dsfont}
\usepackage{amsbsy}
\usepackage{amsmath}
\usepackage{amsfonts,amssymb}
\usepackage{mathrsfs}
\usepackage{mathtools}

\usepackage{tikz}
\usetikzlibrary{calc}

\newcommand{\x}{[x]}

\newcommand{\Prob}{{\rm Prob}}
\newcommand{\Pb}{{\rm P}}
\newcommand{\Pbm}{{\rm P}^m}
\newcommand{\PbG}{{\rm P}^\ast}

\newcommand{\EXm}{{\mathds E}_m}
\newcommand{\EXG}{{\mathds E}_\ast}

\newcommand{\ve}[1]{{\mathbf{#1}}}

\newcommand{\BB}{{\mathds B}}
\renewcommand{\SS}{{\mathds S}}

\newcommand{\e}{{\rm e}}
\newcommand{\X}{\tilde{X}}

\newcommand{\Z}{\tilde{Z}}
\newcommand{\s}{\tilde{S}}
\newcommand{\F}{\mathcal F}
\newcommand{\FG}{\mathcal F_G}
\newcommand{\C}{\mathscr C}
\renewcommand{\O}{\mathcal O}
\renewcommand{\P}{\mathcal P}
\DeclareMathOperator{\Rep}{Rep}

\newcommand{\Ell}{\tau}
\renewcommand{\(}{\left(} 
\renewcommand{\)}{\right)}
\DeclarePairedDelimiter\floor{\lfloor}{\rfloor}

\newtheorem{theorem}{Theorem}[section]
\newtheorem{lemma}{Lemma}[section]
\newtheorem{proposition}{Proposition}[section]
\theoremstyle{definition}
\newtheorem{definition}{Definition}[section]

\theoremstyle{remark}

\makeatletter
\def\namedlabel#1#2{\begingroup
    #2%
    \def\@currentlabel{#2}%
    \phantomsection\label{#1}\endgroup
}
\makeatother

\begin{document}

\title{Brownian Motion in a Vector Space over a Local Field is a Scaling Limit}

\author{
Tyler Pierce$^\ast$, Rahul Rajkumar$^\ast$, Andrea~Stine$^\ast$, David Weisbart$^\ast$ \and Adam M. Yassine$^\dagger$}
\address{\begin{tabular}[h]{cc}
 $^\ast$Department of Mathematics & $^\dagger$Department of Mathematics and Statistics\\
   University of California, Riverside &   Pomona College\\&\end{tabular}}
\email{tpier002@ucr.edu}\email{rrajk004@ucr.edu}\email{astin005@ucr.edu}\email{weisbart@math.ucr.edu}\email{adam.yassine@pomona.edu}

\begin{abstract}
For any natural number $d$, the Vladimirov-Taibleson operator is a natural analogue of the Laplace operator for complex-valued functions on a $d$-dimensional vector space $V$ over a local field $K$.  Just as the Laplace operator on $L^2(\mathds R^d)$ is the infinitesimal generator of Brownian motion with state space $\mathds R^d$, the Vladimirov-Taibleson operator on $L^2(V)$ is the infinitesimal generator of real-time Brownian motion with state space $V$.  This study deepens the formal analogy between the two types of diffusion processes by demonstrating that both are scaling limits of discrete-time random walks on a discrete group.  It generalizes the earlier works, which restricted $V$ to be the $p$-adic numbers. 
\end{abstract}

\maketitle

\begin{flushright}
In memory of Professor V.S.Varadarajan.
\end{flushright}

\tableofcontents

%

\thispagestyle{empty}

\section{Introduction}\label{Sec:Intro}

\subsection{Personal reflection}

D.Weisbart: It feels appropriate to begin this memorial article with a brief reflection on my 
personal relationship with my mentor.  I first met Professor Varadarajan at a faculty-student lunch in 1999, when I was an undergraduate at UCLA; he became my doctoral supervisor a few years later. Professor Varadarajan had a very close relationship with his students.  I remember with great fondness our regular lunches and our many dinners with his wife, Veda. He supported my career and personal growth unwaveringly and did much more for me than one could hope to expect from a mentor. His office and home were always open to his students --- we were family.  His combination of breadth and depth of knowledge was rare, but equally rare was his understanding of the history of our subject and the interplay of its themes through time. His interests were not limited to mathematics; we spoke often about Mozart (a favorite of his), Shakespeare, and even the Old Norse sagas. He was generous, kind, and his qualities as a person give me an ideal to emulate.  

In late December, 2005, I visited Professor Varadarajan to ask for his advice on a problem on which we were collaborating.  He had figured out a key idea only days before.  ``Today," he told me in the serene setting of his backyard, ``is a good day for you because you are a mathematician and you have understood a beautiful idea.  Hopefully, tomorrow will be a better day and you will discover the idea yourself.''  I have long cherished this sentiment.

Professor Varadarajan taught me to appreciate the inner world of mathematics as a sanctuary of the mind, and he encouraged me to pursue mathematics as a part of my personal growth. This perspective has helped me through some difficult times and enriched my experience of our subject.  Since he initially suggested that I study non-Archimedean analogues of physical systems, I am especially grateful to have been able to present to Professor Varadarajan an early result \cite{BW} on the approximation of Brownian motion with a $p$-adic state space by a sequence of discrete time random walks. The present work extends that result and its sequel \cite{WJPA}.  

It is for me a great honor to dedicate this work to the memory of my dear friend and mentor, whose guidance made it possible for me to have the opportunity to co-author this work with my own students. 

\subsection{Mathematical context}

The idea that the ultramicroscopic structure of spacetime might not be locally Euclidean goes back at least to Riemann, who wrote \cite{Riemann:1854} in his 1854 inaugural lecture: %
\begin{quote}
Now it seems that the empirical notions on which the metrical determinations of space are founded, the notion of a solid body and of a ray of light, cease to be valid for the infinitely small. We are therefore quite at liberty to suppose that the metric relations of space in the infinitely small do not conform to the hypotheses of geometry; and we ought in fact to suppose it, if we can thereby obtain a simpler explanation of phenomena.\end{quote} %
Beltrametti and his collaborators \cite{Beltra1, Beltra2, Beltra3} proposed in the 1970's that the geometry of spacetime might involve a non-Archimedean or finite field.  Already in 1975, Bott \cite{Bott75} wrote: %
\begin{quote}
In fact, I would not be too surprised if discrete mod $p$ mathematics and the $p$-adic numbers would eventually be of use in the building of models for very small phenomena.%
\end{quote}

Taibleson introduced the notion of a pseudo-differential operator on complex-valued functions defined on a local field \cite{Taib}.  Saloff-Coste subsequently investigated properties of these operators \cite{SC1} and extended the study to the broader context of local groups \cite{SC2}.  However, the study of physical systems in a non-Archimedean setting really only gained traction following Volovich's seminal article \cite{vol1}.  Varadarajan discussed this contribution of Volovich \cite[Chapter 6]{V2011} and gave some further background on the \emph{Volovich hypothesis} and its relationship to the study of physics using the \emph{Dirac mode}.  For an extensive review of the history of and recent developments in non-Archimedean physics, see the detailed review of Dragovich, Khrennikov, Kozyrev, Volovich, and Zelenov \cite{DKKVZ:2017} that updates the earlier review \cite{DKKV:2009} by the first four authors.

The seminal works of Kochubei \cite{koch92} and Albeverio and Karwowski \cite{alb} on non-Archimedean diffusion followed Vladimirov's works \cite{Vlad88,Vlad90}, in which he studied the operator now known as the Vladimirov operator.  Varadarajan constructed \cite{var97} a general class of diffusion processes with sample paths in the Skorohod space $D([0, \infty)\colon\mathcal S)$ of c\`adl\`ag paths that take values in a finite dimensional vector space $\mathcal S$ with coefficients in a division algebra that is finite dimensional over a non-Archimedean local field of any characteristic.  The present work follows Varadarajan's approach, but specializes the setting to paths in any finite dimensional vector space $V$ over a local field of any characteristic.  Although it does not treat the full generality of Varadarajan's setting, it still broadly generalizes the prior work on scaling limits in the $p$-adic setting \cite{BW, WJPA} by showing that any real-time Brownian motion with paths in $V$ is a scaling limit.  Rather than work with discrete refinements of a primitive, unscaled process that are later embedded into a continuous spacetime, the present work simplifies some aspects of the earlier work by more directly reducing all calculations to involve only a primitive process.

Section~\ref{Sec:Framework} further develops the framework for studying the approximation of Brownian motion that the previous work introduced \cite{WJPA}.  It extends the framework to handle Brownian motion in any finite-dimensional vector space over $\mathds R$ or over a local field $K$.  
Additionally, Section~\ref{Sec:Framework} provides some discussion of the real setting to orient a general readership in a potentially more familiar setting.  Section~\ref{Sec:LimitingProcess} provides a review of analysis in finite dimensional vector spaces over a local field and the construction of Brownian motion in this setting.  Section~\ref{Sec:PrimitiveProcess} constructs the primitive stochastic process whose scaling produces the measures that are the primary objects of interest in this paper.  Section~\ref{Sec:Convergence} establishes uniform moment estimates for the scaled processes.  These estimates are important for proving that the sequence of measures that come from scalings of the primitive process are uniformly tight.  Section~\ref{Sec:Convergence} also establishes the convergence of the finite dimensional distributions of the scaled primitive process to the finite dimensional distributions of the limiting process.  The convergence of the finite dimensional distributions together with the uniform tightness of the sequence of the approximating measures implies Theorem~\ref{Sec:Con:Theorem:MAIN}, the main result.

Although the present article generalizes the earlier work and simplifies the approach, the idea of the proof still follows the idea of the proof in the $p$-adic setting \cite{WJPA}, making this work semi-expository in nature.  The hope is that it will be accessible to a wide readership and encourage further exploration.

\section{A Framework for Approximation}\label{Sec:Framework}

\subsection{Path Spaces}

Take $I$ to be either $[0, \infty)$ or $\mathds N_0$, the natural numbers with $0$.  For any Polish space $\mathcal S$, denote by $F(I\colon \mathcal S)$ the set of all paths from $I$ to $\mathcal S$.  Although proofs of existence of stochastic processes with certain specified probabilities typically involve the construction of a probability measure on $F(I\colon \mathcal S)$, answering analytical questions about Brownian motion requires some restriction to path spaces with sample paths that have some continuity properties.  Central to the present study are the probability measures on the \emph{Skorohod space} $D([0,\infty)\colon \mathcal S)$, the set of c\`{a}dl\`{a}g functions from $[0,\infty)$ to $\mathcal S$ equipped with the Skorohod metric \cite{bil1}.  Take $\Omega(I)$ to be either the space $F(I\colon \mathcal S)$ or $D([0,\infty)\colon \mathcal S)$, and $Y$ to be the function that acts on any pair $(t,\omega)$ in $I\times \Omega(I)$ by \[Y(t,\omega) = \omega(t).\]  Curry variables to view $Y$ as a function that takes each $t$ to the function $Y(t, \cdot)$.  Denote by $\mathcal B(\mathcal S)$ the set of Borel subsets of $\mathcal S$.  Compress notation by writing $Y_t$ rather than $Y(t, \cdot)$, so that for any path $\omega$ in $\Omega(I)$, \[Y_t(\omega) = \omega(t),\] and denote by $\mathcal B(Y)$ the set \[\mathcal B(Y) = \left\{Y_t^{-1}(B)\colon t\in I,\; B\in \mathcal B(\mathcal S)\right\}.\]

Some additional language development is useful.  A \emph{history} $h$ for paths in $\Omega(I)$ is for some natural number $\ell(h)$ a finite sequence $\big((t_0, U_0), (t_1, U_1), \dots (t_{\ell(h)}, U_{\ell(h)})\big)$ with these properties:
\begin{enumerate}
\item[(i)] $t_0 = 0$; 
\item[(ii)] $(t_1, t_2, \dots, t_{\ell(h)})$ is a strictly increasing finite sequence in $I\cap (0,\infty)$;
\item[(iii)] $(U_1, U_2, \dots, U_{\ell(h)})$ is a finite sequence in $\mathcal B(\mathcal S)$.  
\end{enumerate}
The natural number $\ell(h)$ is the \emph{length} of $h$.  The finite sequence of time points in (ii) is the \emph{epoch} for $h$, to be denoted by $e_h$.  The finite sequence of Borel sets in (iii) is the \emph{route} for $h$, to be denoted by $U_h$.  Henceforth, for any $i$ in $\{0, \dots, \ell(h)\}$, use the notation $e_h(i)$ and $U_h(i)$ to designate the $i^{\rm th}$ place of the epoch and route of $h$, respectively.  For any history $h$, the set $\C(h)$ is a \emph{simple cylinder set} that is given by %
\[\C(h) = \bigcap_{i\in \{0, \dots, \ell(h)\}}\{\omega\in\Omega(I)\colon \omega(e_h(i)) \in U_h(i)\}.\] %
Denote by $H$ the set of histories for paths in $\Omega(I)$.  The set of simple cylinder sets is the $\pi$-system $\C(H)$ that generates the $\sigma$-algebra of \emph{cylinder sets}.  Note that any simple cylinder set is a finite intersection of sets in $\mathcal B(Y)$, and that any two probability measures that agree on the simple cylinder sets agree on the cylinder sets as well.  The probabilities associated to the simple cylinder sets are the \emph{finite dimensional distributions} of $Y$.

For any probability measure $\Pb$ on $\Omega(I)$, any $t$ in $I$, and any $B$ in $\mathcal B(\mathcal S)$, the equality%
\[\Pb_t(B) = \Pb(Y_t^{-1}(B))\] %
defines a family $(\Pb_t)_{t\in I}$ of probability measures on $\mathcal S$.  The present work involves paths in a more restrictive setting where $\mathcal S$ is, additionally, a locally compact, Abelian group with Haar measure $\mu$.  In this setting, the family $(\Pb_t)_{t \in I}$ forms a convolution semigroup of probability measures on $\mathcal S$ that gives rise to a measure $\Pb$ on a space of paths in $\mathcal S$.  The finite dimensional distributions of the measure are, for any history $h$, given by %
\begin{align}\label{Framework:EQ:FormulaFDD}
\Pb(\C(h)) = \int_{U_h(0)}\cdots \int_{U_h(\ell(h))} \prod_{i=1}^{\ell(h)}{\rm d}\Pb_{e_h(i)-e_h(i-1)}(x_i-x_{i-1}),
\end{align}  %
and $\Pb(\C(h))$ is nonzero only if $U_h(0)$ contains $0$.  In the case of a Brownian motion on $\mathcal S$, the probability measures come from a convolution semigroup of probability density functions $(\rho(t,\cdot))_{t>0}$, so that \eqref{Framework:EQ:FormulaFDD} becomes %
\begin{align}\label{Framework:EQ:FormulaFDDensity}
\Pb(\C(h)) = \int_{U_h(1)}\cdots \int_{U_h(\ell(h))} \prod_{i=1}^{\ell(h)}\rho(e_h(i)-e_h(i-1), x_i-x_{i-1})\,{\rm d}\mu(x_1)\cdots{\rm d}\mu(x_{\ell(h)}).
\end{align} %
The Kolmogorov extension theorem guarantees that the measures on the simple cylinder sets extend to a measure on the $\sigma$-algebra of cylinder sets of $F(I\colon \mathcal S)$.  

The formal construction of finite dimensional distributions for a process based on some physical intuition about the process typically precedes the verification that there is a process with the proposed finite dimensional distributions.  For this reason, it is useful to refer to the construction of the finite dimensional distributions for a stochastic process as the construction of an \emph{abstract stochastic process} $\tilde{Y}$ and the construction of the law for a random variable without the specification of its domain as the construction of an \emph{abstract random variable} $\tilde{X}$.  A model for the abstract random variable $\tilde{X}$ is a random variable with the same law as $\tilde{X}$.  A model for the abstract stochastic process $\tilde{Y}$ is a stochastic process with the same finite dimensional distributions as $\tilde{Y}$.  To distinguish the probabilities associated with the abstract random variable $\X$ from those associated with a model for $\X$, write $\Prob(\tilde{X}\in A)$ to mean the probability that $\tilde{X}$ takes a value in some Borel set $A$.  Adopt a similar notation for abstract stochastic processes.

It is also helpful to adopt a notation that distinguishes between the discrete-time and continuous-time stochastic processes that the present work involves.  In the discrete-time setting, the time interval is $\mathds N_0$ and the paths are paths in a discrete space, so $F(\mathds N_0\colon \mathcal S)$ is the space of continuous paths.  In this case, write $(F(\mathds N_0\colon \mathcal S), \Pb, S)$ to denote the stochastic process rather than $(\Omega(I\colon \mathcal S), \Pb, Y)$, where for each $n$ in $\mathds N_0$, %
\begin{equation}\label{Framework:Def:SnDef}S_n(\omega) = \omega(n).\end{equation} %
For any history $h$ for paths in $F(\mathds N_0\colon \mathcal S)$, %
\begin{align}\label{Framework:EQ:FormulaFDMass}
\Pb(\C(h)) = \sum_{x_0\in U_h(0)}\cdots \sum_{x_{\ell(h)}\in U_h(\ell(h))} \prod_{i=1}^{\ell(h)}\Pb(S_{e_h(i)}-S_{e_h(i-1)} = x_i-x_{i-1}),
\end{align} %
and is nonzero only if $U_h(0)$ contains $0$.

\subsection{Spatiotemporal embeddings}

Any locally compact topological field is either connected, in which case it is $\mathds R$ or $\mathds C$, or it is totally disconnected.  For this reason, the present convention is that a \emph{local field} is a locally compact, totally disconnected topological field.  

Henceforth, take $m$ to vary in $\mathds N_0$, the field $\mathds F$ to be either $\mathds R$ or a local field $K$, and $V$ to be a normed vector space over $\mathds F$ with norm $\|\cdot\|$.  A countable, discrete Abelian group $(G, +)$ is a \emph{primitive group for a sequence of approximations of $V$} if there is a positive null sequence $(\delta_m)$ and a sequence of injections $(\Gamma_m)$ from $G$ to $V$ with these properties: 
\begin{enumerate}
\item[(1)] For any $v$ in $V$, there is a $g$ in $G$ so that \[\|\Gamma_m(g)-v\| \leq \delta_m;\]
\item[(2)] For any $g_1$ and $g_2$ in $G$, \[\|\Gamma_m(g_1) - \Gamma_m(g_2)\| < \delta_m \implies g_1 = g_2.\]
\end{enumerate}
The sequences $(\delta_m)$ and $(\Gamma_m)$ are, respectively, the \emph{sequence of spatial scales} and the \emph{sequence of spatial embeddings of $G$}.  For any other positive null sequence $(\tau_m)$, a sequence of \emph{spatiotemporal embeddings} of $\mathds N_0\times G$ into $[0,\infty)\times V$ that has $(\Gamma_m)$ as its sequence of spatial embeddings of $G$ and $(\tau_m)$ as its \emph{sequence of time scales} is a sequence of functions $(\iota_m)$ from $\mathds N_0\times G$ to $[0,\infty)\times V$ that is given for each $m$ and each $(n, g)$ in $\mathds N_0\times G$ by \[\iota_m(n,g) = (n\tau_m, \Gamma_m(g)).\]

For any stochastic process $(F(\mathds N_0\colon G), \Pb, S)$ and any $n$ in $\mathds N$, denote by $X_n$ the random variable \[X_n = S_{n+1} - S_n.\]  The $X_n$ are the \emph{increments} of $(F(\mathds N_0\colon G), \Pb, S)$.  The stochastic process $(F(\mathds N_0\colon G), \Pb, S)$ is a \emph{primitive process} in $G$ if $S_0$ is almost surely equal to $0$, and the increments of the process are independent and identically distributed.  Construction of a primitive process involves the specification of a law for an abstractly defined random variable $\X$ that takes values in $G$, and the specification of a sequence of independent abstractly defined random variables $(\X_i)$, each of which have the same law as $\X$.  For any natural number $n$, the equation \begin{equation}\label{Sec:Framework:PrimDefSum}\s_n = \s_0 + \X_1 + \cdots + \X_n\end{equation} determines the abstract stochastic process $\s$.  The Kolmogorov extension theorem guarantees the existence of a stochastic process whose finite dimensional distributions are the probability measures on the simple cylinder sets of $F(\mathds N_0\colon G)$ that $\s$ and \eqref{Framework:EQ:FormulaFDMass} together determine.  Denote this process by $(F(\mathds N_0\colon G), \PbG, S)$.

Use the symbol $S$ ambiguously to identify both the stochastic process and the underlying function on $\mathds N_0 \times F(\mathds N_0\colon G)$.  There is a natural way that $\iota_m$ acts to transform a history $h$ for paths in $F([0,\infty)\colon V)$ to a history $h^m$ for paths in $F(\mathds N_0\colon G)$.  Initially, define $h^m$ to be the finite sequence that is given for any $i$ in $\mathds N_0\cap [0, \ell(h)]$ by \[h^m(i) = \left(\floor*{\tfrac{e_h(i)}{\tau_m}}, \Gamma_m^{-1}(U_h(i))\right).\]  The sequence of time points for $h^m$ may fail to be an epoch since distinct time points that are sufficiently close may map to the same time point.  In such cases, remove all repeated instances of a time point except one, along with their corresponding places in the route. Replace the remaining value of the route, which corresponds to the remaining instance of the repeated time point, with the intersection of all values of the route at the repeated time point.

The functions $\iota_m$ and the process $S$ together produce a collection of measures on the simple cylinder sets of $F([0,\infty)\colon V)$ in the following way:  For any history $h$ for paths in $F([0,\infty)\colon V)$, \begin{equation}\label{EmbeddedProcessFDD}\Pbm(\C(h)) = \PbG(\C(h^m)).\end{equation}  A formal application of \eqref{Framework:EQ:FormulaFDMass} that neglects to take into account the fact that the resulting sequence of time points may no longer be an epoch produces the same probability as \eqref{EmbeddedProcessFDD}.

The Kolmogorov Extension theorem guarantees the existence of a unique probability measure $\Pbm$ on $F([0,\infty)\colon V)$ that restricts to the given measures on the simple cylinder sets.  If the process $(F([0,\infty)\colon V), \Pbm, Y_t)$ has a version with paths in $D([0,\infty)\colon V)$, again use the symbol $\Pbm$ to designate the measure for this process.  Express the relationship between $\iota_m$, $S$, and $\Pbm$ by stating that $\iota_m$ and $S$ induce the measure $\Pbm$ on $D([0,\infty)\colon V)$, with the stochastic process $(D([0,\infty)\colon V), \Pbm, Y_t)$ being the image of $S$ under the spatiotemporal embedding $\iota_m$.  Verification of a certain moment estimate for each $m$ guarantees that there is a version of the process with sample paths in $D([0,\infty)\colon V)$.

In the present setting, to say that a Brownian motion $(D([0,\infty)\colon V), \Pb, Y)$ is a \emph{scaling limit} means that there is a primitive process $(F(\mathds N_0\colon G), \PbG, S)$ together with a sequence of spatiotemporal embeddings $(\iota_m)$ of $\mathds N_0\times G$ into $[0, \infty)\times V$ that induce a sequence of measures $(\Pbm)$ on $D([0,\infty)\colon V)$, and $(\Pbm)$ converges weakly to $\Pb$.  Recall that $(\Pbm)$ converges weakly to $\Pb$ means that for any bounded continuous function $f$ on $D([0,\infty)\colon V)$, \[\int_{D([0,\infty)\colon V)} f(\omega)\,{\rm d}\Pbm \to \int_{D([0,\infty)\colon V)} f(\omega)\,{\rm d}\Pb.\]  Verification that certain moment estimates for $\Pbm$ are uniform in $m$ implies that $(\Pbm)$ is uniformly tight, that is, for any positive $\varepsilon$, there is a compact subset $C_\varepsilon$ of $D([0,\infty)\colon V)$ so that for any $m$ in $\mathds N_0$, $\Pb(C_\varepsilon)$ is in $(1-\varepsilon, 1]$.  Uniform tightness together with the convergence of the finite dimensional distributions of the measures to the finite dimensional distributions of $\Pb$ implies the weak convergence of $(\Pbm)$ to $\Pb$ \cite{bil1, cent}.

\subsection{The classical setting}

Equip $C([0,\infty) \colon \mathds R)$, the set of continuous paths in $\mathds R$ with domain $[0,\infty)$, with the topology of uniform convergence on compacta.  The space of continuous paths is a closed subset of $D([0, \infty)\colon\mathds R)$.  

For any positive real number $D$, the diffusion equation with diffusion constant $D$ is this equation: \begin{equation}\label{Intro:DiffusionEquation}\frac{\partial u}{\partial t} = \frac{D}{2}\frac{\partial^2u}{{\partial x}^2}.\end{equation} The fundamental solution to \eqref{Intro:DiffusionEquation} is the function $\rho$ that for any pair $(t,x)$ in $(0,\infty)\times \mathds R$ is given by \begin{equation}\rho(t,x) = \frac{1}{\sqrt{2\pi Dt}}\exp\left(-\tfrac{x^2}{2Dt}\right).\end{equation} The set $\{\rho(t, \cdot)\colon t\in \mathds R_+\}$ forms a convolution semigroup of probability density functions that determines by \eqref{Framework:EQ:FormulaFDDensity} a probability measure $\Pb$ on $C([0, \infty)\colon\mathds R)$ that is concentrated on the paths that are at $0$ at time $0$.  This measure is the \emph{Wiener measure}.  When studying questions about the approximation of Brownian motion by discrete time random walks, it is helpful to view the measure $\Pb$ as being a measure on $D([0, \infty)\colon\mathds R)$ that gives full measure to $C([0, \infty)\colon\mathds R)$.  It is important to clarify that the Kolmogorov existence theorem initially guarantees that there is a measure on $F([0, \infty)\colon\mathds R)$ with the specified finite dimensional distributions, but the subset of continuous paths (and Skorohod paths as well) is not even a measurable subset of this space.  The moment estimates only guarantee that there is a process with a space of more analytically well-behaved sample paths that has the same finite dimensional distributions.  

Take $\X$ to be an abstract random variable with this law: \[\begin{cases}{\rm Prob}\big(\X = -1\big) = \frac{1}{2}&\mbox{}\\[.2em]{\rm Prob}\big(\X = 1\big) = \frac{1}{2}.&\mbox{}\end{cases}\] Denote by $\s$ the abstract $\mathds Z$-valued stochastic process that $\X$ and \eqref{Framework:EQ:FormulaFDMass} together determine.  Use linearity of the mean and additivity of the variance for sums of independent random variables to obtain for any natural numbers $n$, $n_1$, and $n_2$, with $n_2$ larger than $n_1$, the equalities \[{\mathds E}[\s_n] = 0 \quad \text{and} \quad  {\rm Var}[\s_{n_2} - \s_{n_1}] = n_2-n_1.\]

For any sequence of spatial scales $(\delta_m)$ and time scales $(\tau_m)$, denote by $\Gamma_m$ and $\iota_m$  the functions that for any $(n, z)$ in $N_0\times \mathds Z$ are given by \[\Gamma_m(z) = \delta_mz \quad \text{and} \quad \iota_m(n,z) = (n\tau_m, \Gamma_m(z)).\]  Denote by $\EXm$ the expected value with respect to the measure $\Pbm$ that \eqref{EmbeddedProcessFDD} determines, and obtain for any strictly increasing sequence $(t_1, t_2, t_3)$ in $[0,\infty)$ the inequality \begin{equation}\label{ClassicalDRTMoment}\EXm\big[\left|Y_{t_2}- Y_{t_1}\right|^2\left|Y_{t_3} - Y_{t_2}\right|^2\big] \leq \frac{\delta_m^4}{\tau_m^2}\(t_3 - t_1\)^2,\end{equation} which implies that the process has a version with sample paths in $D([0,\infty)\colon \mathds R)$ \cite{bil1, cent}. Once again use the symbol $\Pbm$ to identify the probability measure for this version, and take $\iota_mS$ to be the process $(D([0,\infty)\colon \mathds R), \Pbm, Y)$.  If there is a positive constant $D$ with \begin{equation}\label{eqn:dtreal}\frac{\delta_m^2}{\tau_m}\to D,\end{equation} then the estimate given by \eqref{ClassicalDRTMoment} is uniform in $m$, which guarantees the uniform tightness of $(\Pbm)$. Uniform tightness and the convergence on the simple cylinder sets of $(\Pbm)$ to $\Pb$ together imply the weak convergence of $(\Pbm)$ to $\Pb$ \cite{bil1,cent}.  Although each $\Pbm$ gives full measure to the $\delta_m\mathds Z$-valued step functions whose jumps occur in $\tau_m\mathds N$ almost surely, the measure $\Pbm$ is concentrated on the continuous functions.

\section{The Limiting Process and its State Space}\label{Sec:LimitingProcess}

\subsection{Vector spaces over local fields}

Any locally compact topological field $K$ is a locally compact, additive Abelian group $(K, +)$, and so has a Haar measure $\mu_K$ that is unique up to scaling.  The multiplicative structure of the field takes any nonzero element $a$ of $(K,+)$ to an automorphism of $K$ that scales the measure of any $\mu_K$-measurable subset by the same factor. Denote this factor by $|a|_K$, so that for any $\mu_K$-measurable set $B$ with finite, nonzero measure, \[|a|_K = \frac{\mu_K(aB)}{\mu_K(B)}.\]  Weil refers to the value $|a|_K$ as the \emph{module} of $a$ \cite{Weil}.  The study of the modules of $K$ is central to the study of topological fields.  The function $|\cdot|_K$ that takes any $a$ in $K$ to its module is an absolute value on $K$ that endows $K$ with the structure of a complete metric space whose topology is the original topology of $K$.  To compress notation, the $K$ in the notation for the absolute value will be suppressed whenever the meaning of the absolute value is unambiguous.

There are two distinct settings:  A local field $K$ is of characteristic 0 or it has positive characteristic.  In the characteristic 0 case, $K$ is either the $p$-adic numbers, $\mathds Q_p$, or a finite field extension of $\mathds Q_p$.  In the positive characteristic case, $K$ is the field $\mathds F_q((t))$, the formal Laurent series over a finite field $\mathds F_q$ with $q$ elements.  The \emph{ring of integers} of $K$ is the maximal proper subring $\O_K$ of $K$, the unit ball in $K$.  In the $\mathds Q_p$ setting, the ring of integers is $\mathds Z_p$, the $p$-adic completion of $\mathds Z$ in $\mathds Q_p$.  The \emph{prime ideal} of $K$ is the ring $\P_K$, which is the unique maximal ideal of $\O_K$.  In the $\mathds Q_p$ setting, the prime ideal is the ring $p\mathds Z_p$.  Any element $\beta$ that satisfies the equality \[\P_K = \beta \O_K\] is a \emph{uniformizer}.   The quotient $\O_K\slash \P_K$ is a finite field with $q$ elements, the \emph{residue field of $K$}, where for some prime $p$, there is a natural number $f$ so that \[q = p^f, \quad |\beta| = \frac{1}{q}, \quad \text{and} \quad |K| = \{0, q^n\colon n\in \mathds Z\}.\]

For any complete set $\Rep(\O_K\slash \P_K)$ of representatives of $\O_K\slash \P_K$, any uniformizer $\beta$ of $K$, and any $x$ in $K$, there is a unique function $a_x$ from $\mathds N_0$ to $\Rep(\O_K\slash \P_K)$ so that \[x = \sum_{k\in \mathds Z} a_x(i)\beta^i.\]  The support of $a_x$ is a bounded below subset of $\mathds Z$ and \[|x| = q^{N},\] where $-N$ is the minimal element in the support of $a_x$.  Denote by $\{\cdot\}$ the function that takes any $x$ in $K$ to its \emph{fractional part}, that is \[\{x\} = \begin{cases}\sum_{k\in -\mathds N} a_x(i)\beta^i &\mbox{if }|x|>1\\0&\mbox{otherwise},\end{cases}\] so that for any $x$ and $y$ in $K$, \[\{x\} = \{y\} \quad \text{if and only if}\quad x-y \in \O_K.\]  

Specification of the norm on a vector space over $K$ follows our earlier study of Brownian motion in more than one $p$-adic variable \cite{RW:JFAA:2023}, but corrects a typographical error in that work. For any natural number $d$, the orthogonal group in $d$ dimensions, $O_d(\mathds R)$, is the maximal compact subgroup of $GL_d(\mathds R)$.  Similarly, for any $d$-dimensional vector space $V$ over $K$, $GL_d(\O_K)$ is the maximal compact subgroup of $GL_d(K)$.  Choose a basis for $V$ to identify $V$ with $K^d$, and take $\|\cdot\|$ to be the max-norm on $K^d$ that is given for any $\ve{x}$ in $K^d$ by \[\|\ve{x}\| = \max_{1\leq i\leq d}(|x_i|), \quad \text{where} \quad \ve{x} = (x_1, \dots, x_d).\] For any $\ve{x}$ in $K^d$ and any $T$ in $GL_d(\O_K)$, the vector $T(\ve{x})$ is determined by multiplying a matrix with coefficients in $\O_k$ by $\ve{x}$, and so \[\|T(\ve{x})\| \leq \|\ve{x}\|.\]  Since $T^{-1}$ is also in $GL_d(\O_K)$, the reverse inequality also holds, which implies Proposition~\ref{GLd-Preserves-norm}.

\begin{proposition}\label{GLd-Preserves-norm}
For any $\ve{x}$ in $K^d$ and any $T$ in $GL_d(\O_K)$, \[\|T\ve{x}\| = \|\ve{x}\|.\] 
\end{proposition}

The $GL_d(\O_K)$ invariance of the max-norm makes it a natural analogue in setting of $d$-dimensional vector spaces over $K$ of the Euclidean norm on $\mathds R^d$, and ensures that the norm of an element is independent of the choice of basis.  With this in mind, henceforth take $V$ to be $K^d$ equipped with the max-norm.

\subsection{The Fourier transform on $K^d$}

A rank 0 additive character $\phi$ for $K$ is a continuous homomorphism from $(K, +)$ to the circle group $\mathds S^1$ with the property that for any integer $n$, \[\int_{\beta^n\O_K} \phi(x)\,{\rm d}\mu(x) = 0 \quad \text{if and only if} \quad n > 0.\]  The additive character $\chi$ on $K$ that is given for any $x$ in $K$ by \[\chi(x) = {\rm e}^{2\pi\sqrt{-1}\{x\}}\] is an example of a rank 0 character.  Take the Haar measure on $K$ to be normalized so that $\mu_K(\O_K)$ is equal to $1$, and henceforth suppress $K$ in the notation for the Haar measure.  The normalized Haar measure $\mu_d$ on $K^d$ gives the unit ball $\O_K^d$ in $K^d$ unit measure and is the product measure induced by $\mu$, so that for any Borel set $B$ in $K$, \[\mu_d(B^d) = (\mu(B))^d, \quad \text{and so} \quad \int_{(\beta^{-n}\O_K)^d} {\rm d}\mu_d(\ve{x}) = q^{nd}.\]  For any $\ve{x}$ and $\ve{y}$ in $K^d$, denote by $\ve{x}\cdot \ve{y}$ the quantity \[\ve{x}\cdot \ve{y} = x_1y_1 +\cdots+x_dy_d, \quad \text{where} \quad \ve{x} = (x_1, \dots, x_d) \quad \text{and} \quad \ve{y} = (y_1, \dots, y_d).\]  For any character $\phi$ that acts on $K^d$, there is a $\ve{y}$ in $K^d$ so that for any $\ve{x}$ in $K^d$, \[\phi(\ve{x}) = \chi(\ve{x}\cdot \ve{y}).\]  Identify $K^d$ with its (isomorphic) Pontryagin dual by the function \[\ve{y} \mapsto \chi(\;\cdot\,\ve{y}), \quad \text{where} \quad \chi(\;\cdot\,\ve{y})(\ve{x}) = \chi(\ve{x}\cdot \ve{y}).\]  The Fourier transform $\F$ on $K^d$ and its inverse are the unitary extensions to $L^2(K^d)$ of the transformations that are initially defined for any $f$ in $L^1(K^d)\cap L^2(K^d)$ by \[(\F f)(\ve{x}) = \int_K \chi(-\ve{x}\cdot \ve{y})f(\ve{y})\mu_d(\ve{y}) \quad \text{and} \quad (\F^{-1}f)(\ve{x}) = \int_K \chi(\ve{x}\cdot \ve{y})f(\ve{y})\mu_d(\ve{y}).\] Note that the choice of normalization of the Haar measure and the choice of $\chi$ as a rank 0 character together imply that $\F$ is a unitary transformation on its initial domain.

\subsection{Brownian motion in $K^d$}

Denote by $\BB_d(n)$ the ball of radius $q^n$ in $K^d$, which is the set $(\beta^{-n}\mathcal O_K)^d$.  Denote by $\SS_d(n)$ the set of all points in $K^d$ with norm equal to $q^n$.  The scaling property of the Haar measure implies that %
\begin{equation}\label{EQ:NormProd:Ball_Sphere_measure_d}
\mu_d(\SS_d(n)) = q^{nd}- q^{(n-1)d} = q^{nd}\left(1 - \tfrac{1}{q^d}\right).
\end{equation} %
To compress notation, write ${\rm d}\ve{x}$ rather than ${\rm d}\mu_d(\ve{x})$ in the notation for integration in $K^d$.  The sum of the $p^{\rm th}$-roots of unity is equal to $0$, and $\chi$ is identically equal to 1 on $\O_K$, so for any $n$ in $\mathds Z$,
\begin{equation}\label{lemma:NormProd:BasicCharInt}
\int_{\BB_1(n)} \chi(x)\,{\rm d}x = \begin{cases}q^{n}&\mbox{if }n\leq 0\\0&\mbox{if }n > 0.\end{cases}
\end{equation}
For any $\ve{y}$ in $K^d$, the scaling property of the measure together with \eqref{lemma:NormProd:BasicCharInt} implies that 
\begin{equation}\label{lemma:NormProd:BasicCharIntd}
\int_{\BB_d(n)} \chi(\ve{x}\cdot \ve{y})\,{\rm d}\ve{x} = \begin{cases}q^{nd}&\mbox{if }q^{n}\leq \frac{1}{\|\ve{y}\|}\\0&\mbox{if }q^{n}> \frac{1}{\|\ve{y}\|}.\end{cases}
\end{equation}
For any subset $S$ of $K^d$, take ${\mathds 1}_S$ to be the indicator function on $S$ and use equation~\ref{lemma:NormProd:BasicCharIntd} to obtain Lemma~\ref{lem:NormProd:CharIntd}

\begin{lemma}\label{lem:NormProd:CharIntd}
For any integer $n$, 
\begin{equation*}
\int_{\BB_d(n)} \chi(\ve{x}\cdot\ve{y}\dwsp)\,{\rm d}\ve{x} = q^{nd}{\mathds 1}_{\BB_d(-n)}(\ve{y}\dwsp).
\end{equation*}
\end{lemma}

Denote by $SB(K^d)$ the space of \emph{Schwartz-Bruhat} functions on $K^d$.  These are the complex-valued, locally constant, compactly supported functions on $K^d$.  The Schwartz-Bruhat functions are dense in $L^2(K^d)$ and, unlike the smooth, compactly supported functions on $\mathds R^d$, they are closed under the Fourier transform.  For any positive real number $b$, take ${\mathcal M}_b$ to be the multiplication operator that acts on $SB(K^d)$ by \[({\mathcal M}_bf)(\ve{x}) = \|\ve{x}\|^bf(\ve{x}).\]  Denote by $\Delta_b^\prime$ the self-adjoint closure of the densely defined, essentially self-adjoint operator that acts on $SB(\mathds K^d)$ by \begin{equation}\label{EQ:NormProd:pseudoDelta}\big(\Delta_b^\prime f\big)(\ve{x}) = \big(\F^{-1}{\mathcal M_b}\F f\big)\!(\ve{x}).\end{equation}  Any operator $T$ that acts on functions in $L^2(K^d)$ extends to act on complex valued functions $f$ on $(0,\infty)\times K^d$ that have the property that for any $t$ in $(0,\infty)$, the function $f(t,\cdot)$ is in the domain of $T$.  The function $Tf$ is the function that is given by \[Tf(t,\ve{x}) = \big(Tf(t,\cdot)\big)(\ve{x}).\]  Take $\Delta_b$ to be the extension of $\Delta_b^\prime$ to the set of all $f$ on $(0,\infty)\times K^d$ with the property that for any positive $t$, $f(t, \cdot)$ is in the domain of $\Delta_b^\prime$.  This extension is the \emph{Taibleson-Vladimirov operator with exponent} $b$.  Extend similarly the Fourier and inverse Fourier transforms to act on functions on $(0,\infty)\times K^d$, but continue to denote them by $\F$ and $\F^{-1}$.  For any positive real number $\sigma$, the pseudo-differential equation \begin{align}\label{EQ:NormProd:DiffusionEQ} \dfrac{{\rm d}f(t,\ve{x})}{{\rm d}t} = -\sigma\Delta_b f(t,\ve{x})\end{align} is a \emph{$d$-dimensional diffusion equation over $K$} and has fundamental solution $\rho$, where for each $t$ in $(0,\infty)$ and for each $\ve{x}$ in $K^d$, \[\rho(t,\ve{x}) = \left(\mathcal F^{-1}\e^{-\sigma t\|\cdot\|^b}\right)\!(\ve{x}).\]  %

Follow Varadarajan's approach \cite{var97} with the necessary modifications to include the diffusion constant $\sigma$ to see that $\rho(t,\cdot)$ is a probability density function that gives rise to a probability measure $\Pb$ on $D([0,\infty) \colon K^d)$ that is concentrated on the set of paths that are at the origin at time $0$.  The \emph{max-norm process} is the stochastic process $\big(D([0,\infty), \Pb, Y\big)$ and is a \emph{Brownian motion} in $K^d$ with \emph{diffusion constant} equal to $\sigma$ and \emph{diffusion exponent} equal to $b$.  Write the integrals of characters over circles as differences of integrals of characters over balls to see that for any $\ve{x}$ in $K^d$, the density function $\rho(t,\cdot)$ for $Y_t$ is given by %
\begin{align}\label{pdf:NormProcess:d}
\rho_d(t,\ve{x}) & = \int_{K^d} \chi(\ve{x}\cdot \ve{y})\e^{-\sigma t\|\ve{y}\|^b}\,{\rm d}\ve{y}\notag\\& = \sum_{n\in\mathds Z}\Big(\e^{-\sigma t q^{nb}} - \e^{-\sigma  t q^{(n+1)b}}\Big)q^{nd}{\mathds 1}_{\BB_d(-n)}(\ve{x}).
\end{align}

\section{The Primitive Process}\label{Sec:PrimitiveProcess}

\subsection{The primitive discrete space}

The ring $\O_K$ is an open compact neighborhood of the identity, and so the quotient group $K\slash\O_K$ is a discrete group.   Take $G$ to be the group $(K\slash\O_K)^d$, the set $(K\slash\O_K)^d$ endowed with componentwise addition.  Denote by $[\cdot]$ the function that takes $K$ to $G$ that is defined for each $(x_1, \dots, x_d)$ in $K^d$ by \[[(x_1, \dots, x_d)] = (x_1 + \O_K, \dots, x_d + \O_K).\]  The measure $\mu_G$ on $G$ is the counting measure, so that for any $B$ in $\mathcal B(K^d)$ that is a union of balls of radius $1$, \[\mu(B) = \mu_G([B]).\] The max-norm on $K^d$ induces a norm on $G$. Namely, for any $\ve{x}$ in $K^d$, \[\begin{cases}\|[\ve{x}]\| = \|\ve{x}\| &\mbox{if }\|\ve{x}\| \ge 1\\\|[\ve{0}]\| = 0.\end{cases}\]  For any integrable function $f$ on $G$, \begin{equation}\label{Eq:PrimRef:IntandSumatm}\int_{G} f([\ve{x}]) \;{\rm d}\mu_G([\ve{x}]) = \sum_{[\ve{x}]\in G} f([\ve{x}]).\end{equation}  

Since any ball in $K^d$ of radius at least $1$ is a disjoint union of balls of radius $1$, the local constancy of any function that is defined on $K^d$ by \[\ve{x} \mapsto f([\ve{x}])\] implies Proposition~\ref{StateSpaces:Intmoverballs}.

\begin{proposition}\label{StateSpaces:Intmoverballs}
For any $\mathds C$-valued integrable function $f$ on $G$ and any ball $B$ of radius at least $1$ in $K^d$,  \[\int_{[B]} f([\ve{x}]) \,{\rm d}\mu_G([\ve{x}]) = \int_B f([\ve{x}]) \,{\rm d}\ve{x}.\]
\end{proposition}

Recall that the group $\O_K^d$ is the Pontryagin dual of $G$.  The canonical inclusion map from $\O_K^d$ to $K^d$ defines the dual pairing $\langle\cdot, \cdot\rangle$ that for any $([\ve{x}],\ve{y})$ in $G\times\O_K^d$ is given by \[\langle [\ve{x}], \ve{y}\rangle = \chi(\ve{x}\cdot \ve{y}).\]  For any $(\ve{x}, \ve{y}, \ve{w})$ in $K^d\times \O_K^d \times \O_K^d$, the equality \begin{align*}\chi((\ve{x}+\ve{w})\cdot \ve{y})= \chi(\ve{x}\cdot\ve{y})\chi(\ve{w}\cdot\ve{y}) = \chi(\ve{x}\cdot\ve{y}) \end{align*} implies that the definition is independent of the choice of the representative. 

The Fourier transform $\FG$ that takes $L^2(G)$ to $L^2(\O_K^d)$ and the inverse Fourier transform $\FG^{-1}$ that takes $L^2(\O_K^d)$ to $L^2(G)$ are, respectively, given for any $f$ in $L^2(G)$ and any $\tilde{f}$ in $L^2(\O_K^d)$ by \[(\FG f)(\ve{y}) = \int_{G}\langle -g,\ve{y}\rangle f([\ve{x}])\,{\rm d}\mu_G(g) \quad \text{and} \quad (\FG^{-1}\tilde{f})\left(g\right) = \int_{\O_K^d}\langle g,\ve{y}\rangle \tilde{f}(\ve{y})\,{\rm d}\ve{y}.\]

\subsection{Transition probabilities of the primitive process}

To simplify notation, for any $i$ in $\mathds N_0$ write \[\BB_G(i) = [\BB_d(i)]\quad \text{and} \quad \SS_G(i) = [\SS_d(i)].\]  Take $\X$ to be the $G$-valued abstract random variable with probability mass function $\rho_{\X}$ that is constant on each circle in $G$ of radius $q^i$ and that satisfies the equalities \begin{equation}\label{equation:PrimitiveRVLaw}\begin{cases} \Prob(\X\in \SS_G(i)) = \frac{q^b-1}{q^{ib}} &\mbox{if }i\in \mathds N\\\Prob(\X=[0]) = 0.&\mbox{}\end{cases}\end{equation}  The law for $\X$ is, therefore, given by \[\begin{cases}\rho_{\X}(\x) = \frac{q^b-1}{q^{ib}}\frac{1}{\mu_G(\SS_G(i))} &\mbox{if } i\in \mathds N\\\rho_{\X}([0]) = 0.&\mbox{}\end{cases}\] It is only slightly modified from the law previously introduced in the $\mathds Q_p$ setting \cite{WJPA} to reflect that the range of the absolute value in $K$ is the set $\{0\}\cup\{q^n\colon n\in \mathds Z\}$ and the fact that the measure of a circle in $K^d$ depends on $d$.  Denote by $\phi_{\X}$ the characteristic function of $\X$.  The determination of $\phi_{\X}$ is useful for determining the transition probabilities for the primitive stochastic process $S$ whose increments have the same law as $\X$.  With this in mind, observe that the domain of $\phi_{\X}$ is $\mathcal O_K^d$, and for any $i$ in $\mathds N$, write \[\rho(i) = \rho_{\X}([\beta^i]) = \left(\frac{q^b-1}{q^{ib}}\right)\left(\frac{q^d}{q^d-1}\right)\frac{1}{\mu_G(\BB_G(i))}\] in order to simplify the calculations that follow.  Use the equality \[\rho(i+1) = \frac{1}{q^{b+d}}\rho(i)\] to see that
\begin{equation}\label{TranProb:Eq:Difference}
\rho(i) - \rho(i+1) = \left(\frac{q^{b+d} -1}{q^{b+d}}\right)\left(\frac{q^b-1}{q^{ib}}\right)\left(\frac{q^d}{q^d-1}\right)\frac{1}{\mu_0(\BB_G(i))}.%
\end{equation}
Proposition~\ref{StateSpaces:Intmoverballs} and \eqref{lemma:NormProd:BasicCharIntd} together imply that
\begin{align}\label{TranProb:Eq:rhozeroint}
\rho(1)\int_{\BB_G([\ve{0}])}\langle -g,\ve{y}\rangle\,{\rm d}\mu_G(g) &= \left(\frac{q^b-1}{q^{b}}\right)\left(\frac{q^d}{q^d-1}\right)\frac{1}{\mu_G(\BB_G(1))}\int_{\BB_G([\ve{0}])}\langle -g,\ve{y}\rangle\,{\rm d}\mu_G(g)\notag\\
&= \left(\frac{q^b-1}{q^{b}}\right)\left(\frac{1}{q^d-1}\right){\mathds 1}_{\O_K^d}(\ve{y}),
\end{align}
and for any natural number $i$, 
\begin{align}\label{TranProb:Eq:rhoiint}
(\rho(i) - \rho(i+1))\int_{\BB_G(i)} \langle -g,\ve{y}\rangle\,{\rm d}\mu_G(g) & = (\rho(i) - \rho(i+1))\mu_G(\BB_G(i)){\mathds 1}_{\beta^i\O_K^d}(\ve{y})\notag\\%
& = \left(\frac{q^{b+d} -1}{q^b}\right)\left(\frac{q^b-1}{q^d-1}\right)\left(\frac{1}{q^{ib}}\right){\mathds 1}_{\beta^i\O_K^d}(\ve{y}).
\end{align}

To simplify notation, write \[\alpha = \frac{1}{q^d-1}\cdot\frac{q^{b+d}-1}{q^b}.\]

\begin{proposition}\label{Prop:CharPrim}
For any $\ve{y}$ in $\O_K^d$, \[\phi_{\X}(\ve{y}) = 1 - \alpha \|\ve{y}\|^b.\] 
\end{proposition}

\begin{proof}
Decompose the integral that defines the Fourier transform of $\rho_{\X}$ into integrals over circles in $G$ to obtain the equalities
\begin{align}\label{Secfour:Equation:PrimYk-sum}
\phi_{\X}(\ve{y}) & = \int_{G}\langle -g,\ve{y}\rangle\rho_{\X}(g)\,{\rm d}\mu_G(g)\notag\\
& = \sum_{i\in \mathds N} \rho(i) \int_{\SS_G(i)} \langle -g,\ve{y}\rangle\,{\rm d}\mu_G(g)\notag\\
& = \sum_{i\in \mathds N} \rho(i)\left\{\int_{\BB_G(i)} \langle -g, \ve{y}\rangle\,{\rm d}\mu_G(g) - \int_{\BB_G(i-1)} \langle -g,\ve{y}\rangle\,{\rm d}\mu_G(g)\right\}\notag\\
& = -\rho(1)\int_{\BB_G(0)}\langle -g,\ve{y}\rangle\,{\rm d}\mu_G(g) + \sum_{i\in \mathds N} (\rho(i) - \rho(i+1))\int_{\BB_G(i)} \langle -g,\ve{y}\rangle\,{\rm d}\mu_G(g).
\end{align}
Use \eqref{TranProb:Eq:rhozeroint} and \eqref{TranProb:Eq:rhoiint} to rewrite \eqref{Secfour:Equation:PrimYk-sum} and obtain the equality
\begin{equation}\label{Secfour:Equation:PrimYk-sumB}
\phi_{\X}(\ve{y}) =  -\left(\frac{q^b-1}{q^{b}}\right)\left(\frac{1}{q^d-1}\right){\mathds 1}_{\mathcal O_K^d}(\ve{y}) + \left(\frac{q^{b+d} -1}{q^b}\right)\left(\frac{q^b-1}{q^d-1}\right)\sum_{i\in \mathds N}\left(\frac{1}{q^{ib}}\right){\mathds 1}_{\beta^i\mathcal O_K^d}(\ve{y}).
\end{equation}
For any $\ve{y}$ in $\mathcal O_K^d$, sum the finite geometric series to obtain the equality \begin{equation}\label{TranProb:Eq:geometric}\sum_{i=1}^{-\log_q\|\ve{y}\|}\frac{1}{q^{ib}} = \frac{1}{q^b-1}(1-\|\ve{y}\|^b).\end{equation} Use \eqref{TranProb:Eq:geometric} to rewrite \eqref{Secfour:Equation:PrimYk-sumB} and obtain the equalities
\begin{align*}
\phi_{\X}(\ve{y}) & = -\left(\frac{q^b-1}{q^{b}}\right)\left(\frac{1}{q^d-1}\right) + \left(\frac{q^{b+d} -1}{q^b}\right)\left(\frac{1}{q^d-1}\right)(1-\|\ve{y}\|^b)\\%
& = 1 - \left(\frac{q^{b+d} -1}{q^b}\right)\left(\frac{1}{q^d-1}\right)\|\ve{y}\|^b.%
\end{align*}
\end{proof}

It is helpful to establish some estimates for the constant $\alpha$.  The dimension $d$ is a natural number and $q$ is at least 2, so%
\begin{equation}\label{alpha:estimateA}
0 < \alpha - 1  = 
\frac{1}{q^d - 1}\frac{q^b - 1}{q^b} <1.
\end{equation}
The inequality
\begin{equation}\label{alpha:estimateDiff}
\frac{{\rm d}}{{\rm d}b}(q^{2b+d} - q^{2b}) > \frac{{\rm d}}{{\rm d}b}(q^{b+d} - 1)
\end{equation}
together with the equality of $q^{2b+d} - q^{2b}$ and $q^{b+d} - 1$ when $b$ is equal to $0$ implies that%
\begin{equation}\label{alpha:estimateB}
\frac{\alpha}{q^b} = \frac{q^{b+d} - 1}{q^{2b+d} - q^{2b}} < 1.
\end{equation} %
Proposition~\ref{Prim:Prop:Boundonalphaoverpb} summarizes these inequalities.

\begin{proposition}\label{Prim:Prop:Boundonalphaoverpb} 
For any natural number $d$ and positive real number $b$, both $\alpha - 1$ and $\frac{\alpha}{q^b}$ are in $(0,1)$.
\end{proposition}

The law for $\X$ together with \eqref{Framework:EQ:FormulaFDMass} and \eqref{Sec:Framework:PrimDefSum} determines a measure $\PbG$ on $F([0,\infty)\colon G)$ and a discrete time stochastic process $(F([0,\infty)\colon G), \PbG, S)$.  For any $n$ in $\mathds N$, denote by $\rho^\ast(n, \cdot)$ the probability mass function for the random variable $S_n$.

\begin{proposition}\label{eq:pmfforprimitiveprocess}
For any pair $(n,g)$ in $\mathds N\times G$, \begin{equation*}%
\rho^\ast(n, g) = (1-\alpha)^n\mathds 1_{\BB_G(0)}(g) + \sum_{i\in \mathds N}\Big(\Big(1-\frac{\alpha}{q^{ib}}\Big)^n - \Big(1-\frac{\alpha}{q^{(i-1)b}}\Big)^n\Big)\frac{1}{\mu_G(\BB_G(i))}\mathds 1_{\BB_G(i)}(g).\end{equation*}
\end{proposition}

\begin{proof}
The Fourier transform takes convolution to multiplication, so Proposition~\ref{Prop:CharPrim} implies that \begin{align}
\rho^\ast(n, g) &= \Big(\mathcal F^{-1}\big(1-\alpha\|\cdot\|^b\big)^n\Big)(g)\notag\\
&= \sum_{i\in \mathds N_0}\int_{\SS_d(-i)}\langle g,\ve{y}\rangle\big(1-\alpha\|\ve{y}\|^b\big)^n\,{\rm d}\ve{y}.
\end{align}
Rewrite integrals over circles in $K^d$ as differences of integrals over balls in $K^d$ to see that
\begin{align}\label{TranProb:Eq:rhotothen}
\rho^\ast(n, g) &=(1-\alpha)^n\bigg(\int_{\BB_d(0)}\langle g,\ve{y}\rangle\,{\rm d}\ve{y} - \int_{\BB_d(-1)}\langle g,\ve{y}\rangle\,{\rm d}\ve{y}\bigg)\notag\\&\hspace{1in} + \sum_{i\in \mathds N}\Big(1-\frac{\alpha}{q^{ib}}\Big)^n\bigg(\int_{\BB_d(-i)}\langle g,\ve{y}\rangle\,{\rm d}\ve{y} - \int_{\BB_d(-(i+1))}\langle g,\ve{y}\rangle\,{\rm d}\ve{y}\bigg)\notag\\
&= (1-\alpha)^n\int_{\BB_d(0)}\langle g,\ve{y}\rangle\,{\rm d}\ve{y}\notag\\&\hspace{1in} + \sum_{i\in \mathds N}\Big(\Big(1-\frac{\alpha}{q^{ib}}\Big)^n - \Big(1-\frac{\alpha}{q^{(i-1)b}}\Big)^n\Big)\int_{\BB_d(-i)}\langle g,\ve{y}\rangle\,{\rm d}\ve{y}.
\end{align}
Use Lemma~\ref{lem:NormProd:CharIntd} to rewrite \eqref{TranProb:Eq:rhotothen} as a scaled sum of indicator functions in order to obtain Proposition~\ref{eq:pmfforprimitiveprocess}.
\end{proof}

\subsection{Moment estimates for the primitive process}

The following theorem of Wendel \cite{Wend} gives estimates that involve the gamma function, $\Gamma$, that are key to the moment estimates that this subsection establishes for the primitive process.  

\begin{theorem}[Wendel's Inequality]\label{WendelsTheorem}
For any positive $x$ and $a$ in $[0,1)$, \[1 \leq \frac{x^a\Gamma(x)}{\Gamma(x+a)} \leq \Big(\frac{x}{x+a}\Big)^{1-a}.\]
\end{theorem}

Denote by $\EXG[\cdot]$ the expected value with respect to the measure $\PbG$ on $F([0,\infty)\colon G)$.%

\begin{theorem}\label{LOM:Theorem:PrimitiveMoments}
For any $r$ in $(0, b)$ and any $n$ in $\mathds N$, 
\begin{equation}\label{eq:KeyMomentEstimatePrimitive}
\EXG\big[\|S_n\|^r\big] < n^{\frac{r}{b}}.\end{equation} 
\end{theorem}

\begin{proof}
Proposition~\ref{eq:pmfforprimitiveprocess} implies that
\begin{align}\label{sec:PrimitiveOrderEstimates}
\EXG\big[\|S_n\|^r\big] &= \int_{G}\|g\|^r\rho^\ast(n,g)\,{\rm d}\mu_G(g)\notag\\
&= \sum_{i\in \mathds N}\Big(\Big(1-\frac{\alpha}{q^{ib}}\Big)^n - \Big(1-\frac{\alpha}{q^{(i-1)b}}\Big)^n\Big)\int_{G}\|g\|^r\frac{1}{\mu_G(\BB_G(i))}\mathds 1_{\BB_G(i)}(g)\,{\rm d}\mu_G(g).
\end{align}
Decompose the integral over $G$ for $\EXG\big[\|S_n\|^r\big]$ into a countable sum of integrals over the circles $\SS_G(i)$ to obtain the equalities %
\begin{align}\label{sec:PrimitiveOrderEstimatesIntegralprelim}
\int_{G}\|g\|^rq^{-id}\mathds 1_{\BB_G(i)}(g)\,{\rm d}\mu_G(g) & = \bigg(\int_{\SS_G(1)}\|x\|^r\,{\rm d}\mu_G(\|g\|) + \cdots + \int_{\SS_G(i)}\|x\|^r\,{\rm d}\mu_G(\|x\|) + 0\bigg)q^{-id}\notag\\
& = \Big(q^r\big(q^d - 1\big) + \cdots + q^{ir}\big(q^{id} - q^{(i-1)d}\big)\Big)q^{-id}\notag\\&= \left(\frac{q^{r+d}-q^r}{q^{r+d}-1}\right)\big(q^{ir} - q^{-id}\big).
\end{align}
Take $C(r,d)$ to be the constant \[C(r,d) = \left(\frac{q^{r+d}-q^r}{q^{r+d}-1}\right),\] so that \eqref{sec:PrimitiveOrderEstimatesIntegralprelim} becomes \begin{equation}\label{sec:PrimitiveOrderEstimatesIntegral}
\int_{G}\|g\|^rq^{-id}\mathds 1_{\BB_G(i)}(g)\,{\rm d}\mu_G(g) = C(r,d)\big(q^{ir} - q^{-id}\big).
\end{equation}

Equalities \eqref{sec:PrimitiveOrderEstimates} and \eqref{sec:PrimitiveOrderEstimatesIntegral} together imply that 
\begin{align}\label{PrimMom:ESNthreetermsb}
\EXG\big[\|S_n\|^r\big] &=\sum_{i\in\mathds N}C(r,d)\big(q^{ir}-q^{-id}\big)\Big(\Big(1-\frac{\alpha}{q^{ib}}\Big)^n - \Big(1-\frac{\alpha}{q^{(i-1)b}}\Big)^n\Big)\notag\\
& = C(r,d)\sum_{i\in\mathds N}q^{ir}\Big(\Big(1-\frac{\alpha}{q^{ib}}\Big)^n - \Big(1-\frac{\alpha}{q^{(i-1)b}}\Big)^n\Big) \notag\\&\hspace{2in}- C(r,d)\sum_{i\in\mathds N}q^{-id}\Big(\Big(1-\frac{\alpha}{q^{ib}}\Big)^n - \Big(1-\frac{\alpha}{q^{(i-1)b}}\Big)^n\Big)\notag\\%
& < C(r,d)\sum_{i>1}q^{ir}\Big(\Big(1-\frac{\alpha}{q^{ib}}\Big)^n - \Big(1-\frac{\alpha}{q^{(i-1)b}}\Big)^n\Big) \notag\\&\hspace{2in} + C(r,d)\left(q^r-q^{-d}\right)\Big(\Big(1-\frac{\alpha}{q^{b}}\Big)^n + \big(\alpha-1\big)^n\Big).
\end{align}
Take $I(n)$ to be the quantity 
 \begin{equation*}
 I(n) = \sum_{i>1}q^{ir}\Big(\Big(1-\frac{\alpha}{q^{ib}}\Big)^n - \Big(1-\frac{\alpha}{q^{(i-1)b}}\Big)^n\Big).
 \end{equation*}
For any natural number $i$ that is greater than 1, \begin{equation}\label{Sec-UBPrim:Eq:DifftonandInt}\Big(1-\frac{\alpha}{q^{ib}}\Big)^n - \Big(1-\frac{\alpha}{q^{(i-1)b}}\Big)^n = n\int_{1-\frac{\alpha}{q^{(i-1)b}}}^{1-\frac{\alpha}{q^{ib}}}s^{n-1}\,{\rm d}s.\end{equation} Since the integrand in \eqref{Sec-UBPrim:Eq:DifftonandInt} is increasing,
\begin{equation*}\label{Sec-UBPrim:Eq:DifftonandIntb}
\Big(1-\frac{\alpha}{q^{ib}}\Big)^n - \Big(1-\frac{\alpha}{q^{(i-1)b}}\Big)^n <n\Big(1-\frac{\alpha}{q^{ib}}\Big)^{n-1}\frac{\alpha}{q^{(i-1)b}}\Big(1-\frac{1}{q^b}\Big),
\end{equation*}
and so
\begin{align}\label{Sec-UBPrim:Eq:toreindex}
I(n) & < \sum_{i>1}q^{ir}n\Big(1-\frac{\alpha}{q^{ib}}\Big)^{n-1}\frac{\alpha}{q^{(i-1)b}}\Big(1-\frac{1}{q^b}\Big)\notag\\
& = n\sum_{i>1}\alpha^{\frac{r}{b}}\Big(\frac{\alpha}{q^{ib}}\Big)^{-\frac{r}{b}}\Big(1-\frac{\alpha}{q^{ib}}\Big)^{n-1}\frac{\alpha}{q^{ib}}\big(q^b-1\big).
\end{align}
Following the calculation in the $\mathds Q_p$ setting \cite{WJPA}, write \[x_i = \frac{\alpha}{q^{ib}}, \quad I_i = [x_i, x_{i-1}], \quad {\rm and}\quad \Delta I_i = x_{i-1} - x_{i}\] to obtain the equality \begin{equation}\label{Sec-UBPrim:SimpWithXiIi}n\sum_{i>1}\alpha^{\frac{r}{b}}\Big(\frac{\alpha}{q^{ib}}\Big)^{-\frac{r}{b}}\Big(1-\frac{\alpha}{q^{ib}}\Big)^{n-1}\frac{\alpha}{q^{ib}}\big(q^b-1\big) = n\alpha^{\frac{r}{b}}\sum_{i>1}x_i^{-\frac{r}{b}}(1-x_i)^{n-1}\Delta I_i.\end{equation}
Use the fact that $\Delta I_i$ is equal to $q^b\Delta I_{i+1}$ and reindex the sum in the righthand side of \eqref{Sec-UBPrim:SimpWithXiIi} to see that 
\begin{align}\label{Sec-UBPrim:InpreBetafirst}
I(n) & < n\alpha^{\frac{r}{b}}q^{b}\sum_{i>2}x_{i-1}^{-\frac{r}{b}}\big(1-x_{i-1}\big)^{n-1}\Delta I_i,
\end{align}
from which follows the inequality
\begin{align}\label{Sec-UBPrim:InpreBeta}
I(n)& <  n\alpha^{\frac{r}{b}}q^{b}\int_0^{\frac{\alpha}{q^{2b}}}x^{-\frac{r}{b}}(1-x)^{n-1}\Delta I_i
\end{align}
since the sum in \eqref{Sec-UBPrim:InpreBetafirst} is a lower Riemann sum approximation of the integral in \eqref{Sec-UBPrim:InpreBeta}.  The positivity on $(0,1)$ of the integrand in \eqref{Sec-UBPrim:InpreBeta} implies that
\begin{align}\label{Sec-UBPrim:InpreBetab}
I(n) & <  n\alpha^{\frac{r}{b}}q^{b}\int_0^{1}x^{-\frac{r}{b}}(1-x)^{n-1}\,{\rm d}x\notag\\ 
& = n\alpha^{\frac{r}{b}}q^{b}B\Big(\frac{b-r}{b}, n\Big) = n\alpha^{\frac{r}{b}}q^{b}\frac{\Gamma\big(\frac{b-r}{b}\big)\Gamma(n)}{\Gamma\big(\frac{b-r}{b}+n\big)},
\end{align}
where $B$ is the beta function and $\Gamma$ is the gamma function. Wendel's inequality implies that \begin{equation}\label{Wend:Gamma}\frac{\Gamma(n)}{\Gamma\big(\frac{b-r}{b}+n\big)}\leq n^{-\frac{b-r}{b}}\Big(\frac{n+\frac{b-r}{b}}{n}\Big)^{1-\frac{b-r}{b}} = n^{-\frac{b-r}{b}}\Big(\frac{n+\frac{b-r}{b}}{n}\Big)^{\frac{r}{b}} = e_1(n)n^{-\frac{b-r}{b}}\end{equation}
where for all $n$, \[e_1(n) < 2 \quad {\rm and}\quad \lim_{n\to \infty} e_1(n) =1.\] The inequalities \eqref{Sec-UBPrim:InpreBetab} and \eqref{Wend:Gamma} together imply that
\begin{align}\label{Sec-UBPrim:InBetaalmostthere}
I(n) &< %
2\alpha^{\frac{r}{b}}q^b\Gamma\Big(\frac{b-r}{b}\Big)n^{\frac{r}{b}}.
\end{align}

To bound from above the second summand in  \eqref{PrimMom:ESNthreetermsb}, take \[\delta = \max\Big\{\Big|1-\frac{\alpha}{q^b}\Big|, |1-\alpha|\Big\} < 1\] so that \begin{equation}\label{PrimMom:equation:seconderrorterm} 0 < \bigg|\Big(1-\frac{\alpha}{q^b}\Big)^n + (1-\alpha)^n\bigg| \leq 2\delta^n = e_2(n)n^{\frac{r}{b}} \quad {\rm where}\quad e_2(n) = (2\delta^n n^{-\frac{r}{b}}).\end{equation} Since $e_2(n)n^{\frac{r}{b}}$ tends to zero as $n$ tends to infinity and $e_2(n)$ is bounded above by 2, the inequalities \eqref{PrimMom:ESNthreetermsb}, \eqref{Sec-UBPrim:InBetaalmostthere}, and \eqref{PrimMom:equation:seconderrorterm} together imply that
\begin{align*}
\EXG\big[\|S_n\|^r\big] &< Kn^{\frac{r}{b}}, \quad \text{where} \quad K = 2C(r,b)\Big(\alpha^{\frac{r}{b}}q^b\Gamma\Big(\frac{b-r}{b}\Big) + \left(q^r-q^{-d}\right)\Big).
\end{align*}

\end{proof}

\section{Convergence of the Approximations}\label{Sec:Convergence}
 
The method for calculating moment estimates of the primitive process closely follows the approach of the earlier work \cite{WJPA}. However, the method of calculation in this present section is different.  Here, a more streamlined approach involves direct embeddings of $G$ into $K^d$ and the reduction of all calculations to involve only the primitive process in $G$.

\subsection{Embeddings of the primitive process}

For any $[\ve{x}]$ in $(K\slash \mathcal O_K)^d$, take $x_j$ to be a representative of the $j^{\rm th}$ component of $[\ve{x}]$.  Define the function $\Gamma_0$ by \[\Gamma_0([\ve{x}]) = \left(\sum_{i\in \mathds N} a_{x_j}(i)\beta^i\right)_{j\in\{1, \dots, d\}},\] which is independent of the choice of representative.  Denote by $\Gamma_m$ the injective function \[\Gamma_m([\ve{x}]) = \beta^m\Gamma_0([\ve{x}]).\]  For any complex-valued function $f$ on $G$, the function $\Gamma_m$ determines a locally constant function $f_m$ on $K^d$ in the following way:  For any $[\ve{x}]$ in $G$, take $f_m$ to be the function that is defined for any $\ve{z}$ in $\Gamma_m([\ve{x}])+(\beta^m\O_K)^d$ by \[f_m(\ve{z}) = q^{md}f([\ve{x}]).\] %
Note that $f_m$ has radius of local constancy equal to $q^{-m}$, and for any ball $B$ in $K^d$, if the radius of $B$ is at least $q^{-m}$, then \begin{equation}\label{Convergence:Equation:fmintA}\int_{[B]} f([\ve{x}])\,{\rm d}\ve{x} = \int_{\beta^mB} f_m(\ve{x})\,{\rm d}\ve{x},\end{equation} and so if $f$ is in $L^1(G)$, then $f_m$ is in $L^1(K^d)$, and \begin{equation}\label{Convergence:Equation:fmintB}\int_G f([\ve{x}])\,{\rm d}\mu_G([\ve{x}]) = \int_{K^d} f_m(\ve{x})\,{\rm d}\ve{x}.\end{equation}

\renewcommand{\Ell}{\lambda}


\subsection{Moment estimates for the embedded processes}

Denote by $\Pbm$ the measure on $F([0,\infty)\colon K^d)$ with finite dimensional distributions given by \eqref{EmbeddedProcessFDD}. Denote by $\EXm[\cdot]$ the expected value with respect to $\Pbm$.  Take $(\tau_m)$ to be the sequence that is given for some positive real number $C^\prime$ by $(C^\prime\delta_m^b)$.  For typographical purposes, denote by $\Ell(m)$ the reciprocal of $\tau_m$, and for any $t$ in $[0, \infty)$, take $t_m$ to be the real number $\floor*{\Ell(m)t}$.

\begin{proposition}\label{MainEstimateEmbedded}
There is a constant $C$ such that for any $t$ in $[0, \infty)$, \[\EXm\!\left[\|Y_t\|^r\right] \leq Ct^{\frac{r}{b}}.\]
\end{proposition}

\begin{proof}

For any $t$ that is small enough so that $t_m$ is equal to $0$, the inequality is valid because $\EXm\!\left[\|Y_t\|^r\right]$ is equal to $0$.  If $t\Ell(m)$ is in $[1, \infty)$, then \eqref{EmbeddedProcessFDD} implies that
\begin{align*}
\EXm[\|Y_t\|^r] & = %
%
\sum_{k\in \mathds Z} (q^{-m}q^k)^r\Pbm(\|Y_t\| = q^{-m}q^k)\\
& = q^{-mr}\sum_{k\in \mathds Z} (q^k)^r\PbG(\|S_{t_m}\| = q^k) = q^{-mr}\EXG[\|S_{t_m}\|^r].
\end{align*} %
Theorem~\ref{LOM:Theorem:PrimitiveMoments} implies that there is a constant $K$ so that
\begin{align*}
{\EXm}[\|Y_t\|^r] 
& \leq Kq^{-mr}t_m^{\frac{r}{b}}.
\end{align*}
The estimate to be proved then follows from the inequality 
\begin{align*}
t_m^{\frac{r}{b}} \leq \left(\lambda_mt\right)^{\frac{r}{b}}.
\end{align*}

\end{proof}

\begin{proposition}\label{tight}
The stochastic process $(F([0,\infty)\colon K^d), \Pbm, Y)$ has a version with paths in $D([0,\infty)\colon K^d)$, a process $(D([0,\infty)\colon K^d), \Pbm, Y)$.  The sequence of measures $(\Pbm)$ with paths in $D([0,\infty)\colon K^d)$ is uniformly tight.
\end{proposition}

\begin{proof}
For any strictly increasing finite sequence $(t_1, t_2, t_3)$ in $[0, \infty)$, the independence of the increments of the process $(F(\mathds N_0\colon G), \PbG, S)$ implies that 
\begin{align*}
\EXm\!\left[\big|Y_{t_3} - Y_{t_2}\big|^r\big|Y_{t_2} - Y_{t_1}\big|^r\right] & = \EXm\!\left[\big|Y_{t_3} - Y_{t_2}\big|^r\right]{\mathds E}_m\!\left[\big|Y_{t_2} - Y_{t_1}\big|^r\right]\\ & \leq C(t_3-t_2)^\frac{r}{b}C(t_2-t_1)^\frac{r}{b} \leq C^2(t_3-t_1)^{\frac{2r}{b}}.
\end{align*}
Take $r$ to be any real number in $\left(\frac{b}{2}, b\right)$ to verify that $(F([0,\infty)\colon K^d), \Pbm, Y)$ satisfies the criterion of Chentsov, which implies that the stochastic process $(F([0,\infty)\colon K^d), \Pbm, Y)$ has a version with paths in $D([0,\infty)\colon \mathds Q_p)$.  Since the constant $C$ is independent of both $m$ and the choice of $(t_1, t_2, t_3)$, the sequence $(\Pbm)$ of probability measures is uniformly tight \cite{cent}.  
\end{proof}

Proposition~\ref{prop:qp:concenration} demonstrates that, just as in the $\mathds Q_p$ setting, the sequence of stochastic processes given by $(D([0,\infty)\colon K^d), \Pbm, Y)$ describes random walks on the discrete subsets $\Gamma_m(G)$ of $K^d$, with jumps that occur only at time points that are positive multiples of $\tau_m$.  %

\begin{proposition}\label{prop:qp:concenration}
The measure $\Pbm$ is concentrated on the subset of $\Gamma_m(G)$-valued paths in $D([0,\infty)\colon K^d)$ that are constant on each interval in $\big\{\big[(n-1)\tau_m, n\tau_m\big)\colon n\in\mathds N\big\}$.
\end{proposition}

\begin{proof}
For any history $h_\ast$, denote by $T(h_\ast)$ the set of time points in the epoch for $h_\ast$.  Take $(h_n)$ to be any sequence of histories for paths in $D([0, \infty)\colon K^d)$ whose routes at every nonzero place take on the value $\Gamma_m(G)$, and whose epochs have the property that $(T(h_n))$ is a nested sequence so that $\cup_{n\in \mathds N} T(h_n)$ is dense in $[0,\infty)$.  %
The equality \[\Pbm(C(h_n)) = \PbG\Big(\bigcap_{i\in \{0, \dots, \ell(h_n)\}} S_{\floor*{\frac{e_{h_n}(i)}{\tau_m}}}^{-1}(G)\Big) = 1\] together with the continuity from above of $\Pbm$ implies that $\Pbm$ gives full measure to the paths that are valued in $\Gamma_m(G)$ on a dense subset of time points.  Right continuity of the paths implies that these paths are $\Gamma_m(G)$-valued.

For any subinterval $I$ of $[0,\infty)$ so that $I\cap \tau_m\mathds N$ is empty, take $(V_i)$ to be any nested sequence of finite subsets of $I$ whose union is $I\cap \mathds Q$.  The equality for any real numbers $s_1$ and $s_2$ in $I$ of $\floor*{s_1\Ell(m)}$ and $\floor*{s_2\Ell(m)}$ implies that\begin{align*} \Pbm\big(Y_{s_1} - Y_{s_2} = 0\big) = \PbG\big(S_{\floor*{s_1\Ell(m)}} - S_{\floor*{s_2\Ell(m)}} = 0\big) = 1.\end{align*}  Continuity from above of the measure $\Pbm$ implies that for any $t$ in $I$, \begin{align*}&\Pbm\!\(\left\{\omega\in D([0,\infty)\colon K^d)\colon |\omega(t) - \omega(s)| = 0, \quad \forall s \in I\cap \mathds Q\right\}\) \\&= \lim_{i\to \infty} \Pbm\!\(\left\{\omega\in D([0,\infty)\colon K^d)\colon |\omega(t) - \omega(s)| = 0, \quad \forall s \in I\cap V_i\right\}\) = 1.\end{align*}  Right continuity of the paths implies that the paths are $\Pbm$-almost surely constant on $I$.  Take $(I_n)$ to be any sequence of disjoint intervals that do not intersect $\tau_m\mathds N$ and whose union is $[0,\infty)\setminus \tau_m\mathds N$.  The set of paths that are constant on every interval in $[0,\infty)\setminus \tau_m\mathds N$ is the countable intersection of the set of paths that are constant on each of the $I_n$. Continuity from above of the measure $\Pbm$ therefore implies that $\Pbm$ is concentrated on the set of paths that are constant on every interval in $[0,\infty)\setminus \tau_m\mathds N$. 
\end{proof}


\subsection{Convergence of the processes}

Denote by $E_m$ the function that is defined for any $(t,\ve{x})$ in $(0,\infty)\times K^d$ by \[E_m(t, \ve{x}) = \begin{cases}  
\left(1 - \frac{\alpha\|\ve{x}\|^b}{q^{mb}}\right)^{\floor{t\Ell(m)}} &\mbox{if }\|\ve{x}\| \leq q^{m}\\0 &\mbox{if }\|\ve{x}\| > q^{m}.\end{cases}\]

\begin{lemma}\label{last:Emlimit}
For any $t$ in $(0, \infty)$, \[\lim_{m\to \infty} \int_{K^d}\left|E_m(t,\ve{y}) - \e^{-\sigma \|\ve{y}\|^bt}\right|\,{\rm d}\ve{y} = 0.\]
\end{lemma}

\begin{proof}
For any $z$ in $[0, 1]$, the inequality \begin{equation*}0 \leq 1-z \leq \e^{-z}\end{equation*} implies that for any $m$ and any $x$ in $[0, m]$, \begin{equation}\label{ConofProcess:expineq:Dini}0 \leq \left(1-\frac{x}{m}\right)^m \leq \left(\e^{-\frac{x}{m}}\right)^m = \e^{-x}.\end{equation}  Since the sequence on the lefthand side of \eqref{ConofProcess:expineq:Dini} is increasing, for any positive real number $R$, Dini's theorem implies that $\left(1-\frac{x}{m}\right)^m$ converges uniformly to $\e^{-x}$ in $[0, R]$.  Proposition~\ref{Prim:Prop:Boundonalphaoverpb} guarantees that if $\|\ve{x}\|$ is less than $q^m$, then \begin{equation}\label{COP:Lem:expext}\alpha\|\ve{x}\|^b \leq q^{mb},\end{equation} and so \eqref{ConofProcess:expineq:Dini} together with the integrability of the exponential function $(-\infty, 0]$ implies that there is a divergent, increasing, positive sequence $(R_m)$ that is strictly bounded above by $q^m$ and \begin{equation}\label{COP:Lem:expextSeq}\lim_{m\to \infty} \int_{\|\ve{y}\| \leq R_m}\left|E_m(t,\ve{y}) - \e^{-\sigma \|\ve{y}\|^bt}\right|\,{\rm d}\ve{y} = 0.\end{equation}  Denote by $A_m$ the set \[A_m = \{\ve{y} \in K^d\colon R_m<\|y\|<q^m\; \text{and} \; \|y\| > q^m\}.\]  The ball $\BB_d(m)$ is the support of $E_m$, so the inequality \eqref{COP:Lem:expext} implies that for any $y$ in $A_m$, \[\left|E_m(t,\ve{y}) - \e^{-\sigma \|\ve{y}\|^bt}\right| \leq \e^{-\sigma \|\ve{y}\|^bt},\] and so \begin{equation}\label{COP:Lem:expextFar}\lim_{m\to \infty} \int_{A_m}\left|E_m(t,\ve{y}) - \e^{-\sigma \|\ve{y}\|^bt}\right|\,{\rm d}\ve{y} = 0.\end{equation}  Since $|1-\alpha|$ is in $[0,1)$, \begin{align}\label{COP:Lem:expextSingle}\lim_{m\to \infty} \int_{\SS_m}\left|E_m(t,\ve{y}) - \e^{-\sigma \|\ve{y}\|^bt}\right|\,{\rm d}\ve{y} &\leq \lim_{m\to \infty} q^{md}\left|(1 - \alpha)^{\floor{t\Ell(m)}} - \e^{-\sigma q^{md}t}\right| = 0.\end{align}  Decompose the integral in the statement of the lemma into three regions, the ball of radius less than $R_m$, the set $A_m$, and the circle of radius $q^m$ and use \eqref{COP:Lem:expextSeq}, \eqref{COP:Lem:expextFar}, and \eqref{COP:Lem:expextSingle} to obtain the desired limit. 
\end{proof}

The set $H_R$ of \emph{restricted histories} for paths in $D([0, \infty)\colon K^d)$ is the set of all histories whose route is a finite sequence of balls.  

\begin{proposition}\label{5:prop:restrictedhistconv}
For any restricted history $h$ in $H_R$, \[P^m(\C(h)) \to P(\C(h)).\]
\end{proposition}

\begin{proof}

For any ball $B$ of radius $q^{-m}$ in $K^d$, if $\ve{x}$ is in $B$, then $[\beta^{-m}\ve{x}]$ is the unique element of $G$ so that $\Gamma_m([\beta^{-m}\ve{x}])$ is in $B$, and so \eqref{EmbeddedProcessFDD} and \eqref{Convergence:Equation:fmintA} together imply that %
\begin{align*}
\Pbm(Y_t\in B) &= \PbG\big(S_{t_m}\in \Gamma_m^{-1}(B)\big)\\ %
& = \rho^\ast\big(t_m,\beta^{-m}\ve{x}\big)\\ %
& = q^{md}\int_{\ve{x} + \beta^m\O_K^d}\rho^\ast\big(t_m,[\ve{z}]\big)\,{\rm d}\ve{z}\\%
& = q^{md}\int_{K^d}\rho^\ast\big(t_m,[\beta^{-m}\ve{x}]\big)\mathds 1_{\ve{x} + \beta^m\O_K^d}(\ve{z})\,{\rm d}\ve{z}.
\end{align*}
Denote by $\rho_m$ the function that for each $t$ in $(0,\infty)$ is given by %
\begin{equation}
\rho_m(t, \ve{z}) = \sum_{[\beta^{-m}\ve{x}]\in G}q^{md}\rho^\ast\big(t_m,[\beta^{-m}\ve{x}]\big)\mathds 1_{\ve{x}+ \beta^m\O_K^d}(\ve{z}), %
\end{equation}
so that for any $t$ in $(0,\infty)$ and any set $U$ in $K^d$ that is a finite union of balls that each have radius at least $q^{-m}$, \[P^m(Y_t\in U) = \int_{K^d}\rho_m(t,\ve{x})\,{\rm d}\ve{x}.\]

The probability mass function $\rho^\ast$ for the primitive process is given by the equality %
\begin{align*}
\rho^\ast(n, [\ve{x}]) &= \int_{\O_K^d} \langle [\ve{x}], \ve{y}\rangle \big(1-\alpha \|\ve{y}\|^b\big)^n\,{\rm d}\ve{y},
\end{align*}
which implies that 
\begin{align*}
\rho_m(t, \Gamma_m([\ve{x}])) & = q^{md}\int_{\O_K}\chi(\beta^{-m}\ve{x}\cdot\ve{y})(1-\alpha\|\ve{y}\|^b)^{t_m}\,{\rm d}\ve{y}.%
\end{align*}
Use the change of variables \[\ve{z} = \beta^{-m}\ve{y}\]%
to obtain the equalities
\begin{align*}
\rho_m(t, \Gamma_m([\ve{x}])) & = q^{md}\int_{\beta^{-m}\O_K}q^{-md}\chi(\ve{x}\cdot\ve{z})\left(1-\alpha\frac{\|\ve{z}\|^b}{q^{mb}}\right)^{t_m}\,{\rm d}\ve{z}\\
& = \int_{\beta^{-m}\O_K}\chi(\ve{x}\cdot\ve{z})E_m(\ve{z})\,{\rm d}\ve{z} = \int_{K^d}\chi(\ve{x}\cdot\ve{z})E_m(\ve{z})\,{\rm d}\ve{z}.
\end{align*}
Take absolute values of the difference between $\rho_m$ and $\rho$ to see that 
 \begin{align}\label{5:prop:restrictedhistconv:unlim}
\left|\rho_m(t, \ve{x}) - \rho(t, \ve{x})\right| & \leq \int_{K^d} \left|\chi(\ve{x}\cdot\ve{y}) E_m(t,\ve{y}) - \chi(\ve{x}\cdot\ve{y}){\rm e}^{-\sigma t\|\ve{y}\|^b}\right|\,{\rm d}\ve{y}\notag\\& = \int_{K^d} \left|E_m(t,\ve{y}) - {\rm e}^{-\sigma t\|\ve{y}\|^b}\right|\,{\rm d}\ve{y} = \varepsilon_m(t) \to 0.
\end{align}
The sequence $(\varepsilon_m(t))$ is independent of $\ve{x}$, which implies that %
$(\rho_m(t, \cdot))$ converges uniformly on $K^d$ to $\rho(t, \cdot)$.

Take $h$ to be any restricted history and without loss in generality suppose that $U_h(0)$ is the set $\{0\}$.  Simplify the notation by writing \[e_h = (t_0, \dots, t_k) \quad {\rm and}\quad U_h = (\{0\}, U_1, \dots, U_k).\]  For any $i$ in $\{0, \dots, k\}$, denote by $r_i$ the radius of $U_i$. For any $m$ so that \[q^{-m} < \min\{r_1, \dots, r_k\},\] the uniform convergence given by \eqref{5:prop:restrictedhistconv:unlim} implies that %
\begin{align*}
\Pbm(\C(h)) &= \Pbm(Y_{t_1}\in U_1, Y_{t_2}\in U_2, \dots, Y_{t_n}\in U_n)\\
&= \int_{U_1} \cdots \int_{U_k} \prod_{i\in\{1, \dots, k\}}\rho_m(t_i-t_{i-1}, \ve{x}_i -\ve{x}_{i-1})\,{\rm d}\ve{x}_k\cdots {\rm d}\ve{x}_1\\
&\to  \int_{U_1}\cdots \int_{U_k} \prod_{i\in\{1, \dots, k\}}\rho(t_i-t_{i-1}, \ve{x}_i-\ve{x}_{i-1})\,{\rm d}\ve{x}_k\cdots {\rm d}\ve{x}_1 = \Pb(\C(h)).
\end{align*}
\end{proof}

The intersection of any two balls in $K^d$ is again a ball in $K^d$, and so $\C(H_{R})$ is a $\pi$-system that generates the $\sigma$-algebra of cylinder sets of paths in $D([0, \infty)\colon K^d)$.  The uniform tightness of the family of measures $\{\Pbm\colon m\in \mathds N_0\}$ that Proposition~\ref{tight} guarantees together with the convergence for any restricted history $h$ of $(\Pbm(\C(h)))$ to $\Pb(\C(h))$ implies Theorem~\ref{Sec:Con:Theorem:MAIN}.

\begin{theorem}\label{Sec:Con:Theorem:MAIN}
The sequence of measures $(\Pbm)$ converges weakly to $\Pb$.
\end{theorem}

As in the real case and the $p$-adic case \cite{BW, WJPA}, the scaling factor for the time scales is proportionate to a power of the scaling factor for the space scales.  The power is the exponent of the Vladimirov-Taibleson operator, and is independent of both $q$ and the dimension.

\end{document}